\documentclass[12pt]{article}
\usepackage{amssymb, amsmath, amsthm, amsfonts, tikz, geometry, ytableau, mathrsfs}
\usepackage[indent,margin=1cm]{caption}
\usepackage{fullpage}
\usetikzlibrary{decorations.pathreplacing}
\usepackage[colorlinks,linkcolor=blue,anchorcolor=blue,citecolor=blue]{hyperref}

\numberwithin{equation}{section}
\numberwithin{figure}{section}

\hypersetup{colorlinks = true}
\newtheorem{thm}{Theorem}[section]
\newtheorem{conj}[thm]{Conjecture}

\newtheorem{cor}[thm]{Corollary}

\newtheorem{prop}[thm]{Proposition}
\newtheorem{prob}[thm]{Problem}

\newtheorem{exam}[thm]{Example}
\newtheorem{rem}[thm]{Remark}
\allowdisplaybreaks

\begin{document}
\begin{center}
	{\large \bf A quantitative way to $e$-positivity of trees}
\end{center}

\begin{center}
Ethan Y.H. Li
\\[6pt]
\end{center}

\begin{center}
School of Mathematics and Statistics,\\
Shaanxi Normal University, Xi'an, Shaanxi 710119, P. R. China \\[6pt]
Email: 
{\tt yinhao\_li@snnu.edu.cn}
\end{center}

\noindent\textbf{Abstract.}
In 2020, Dahlberg, She, and van Willigenburg conjectured that the chromatic symmetric function of any tree with maximum degree at least 4 is not $e$-positive. Zheng and Tom verified this conjecture for all trees with maximum degree at least 5 and spiders with maximum degree 4, and in their proofs the following necessary condition given by Wolfgang plays an important role: every connected graph having $e$-positive chromatic symmetric function must contain a connected partition of every type.

In order to make further progress on this conjecture, we refine Wolfgang's result in a quantitative way. At first, we give an explicit formula for the $e$-coefficients of trees in terms of their connected partitions, by which $e$-positivity is equivalent to a series of inequalities for the numbers of connected partitions. Based on this formula, we present several necessary conditions on the numbers of connected partitions or acyclic orientations for trees to be $e$-positive. These necessary conditions turn out to be characterizations on the structure of $e$-positive trees, and as sample applications we prove the non-$e$-positivity of several classes of trees with maximum degree 3 or 4. We further make more discussions and calculations on trees with maximum degree 4 and having a connected partition of every type, which inspire us to come up with a list of open problems towards the final resolution of the above conjecture.\\



\section{Introduction}\label{sec-intro}

Given a graph $G$ with $V(G) = \{v_1,\ldots,v_d\}$, the \textit{chromatic symmetric function} of $G$ is defined as
\begin{equation*}
X_G = \sum_{\kappa} x_{\kappa(v_1)}\cdots x_{\kappa(v_d)},
\end{equation*}
where $\kappa: V(G) \to \{1,2,\ldots\}$ ranges over all proper colorings of $G$, i.e., $\kappa(u) \neq \kappa(v)$ for any edge $uv \in E(G)$. This concept was introduced by Stanley \cite{Sta95} in 1995 as a symmetric function generalization of the famous chromatic polynomial $\chi_G(n)$, which can be obtained by setting $x_1=x_2=\cdots =x_n = 1$ and $x_{n+1} = x_{n+2} = \cdots = 0$.

The famous Stanley-Stembridge conjecture states that the incomparability graph of each $(3+1)$-free poset is a non-negative linear expansion of elementary symmetric functions (also called $e$-positive), which is equivalent to a conjecture on immanants of Jacobi-Trudi matrices \cite{Sta95,SS93}. During these decades, there have been plenty of works on this conjecture, see \cite{BC18,DFvW20,DSvW20,Gas96,GS01,G-P13,Hai93,SW16} and the references therein. In particular, Hikita firstly proved this conjecture \cite{Hik24}, and later Griffin et al. provided another proof \cite{GMRWW25}. These incomparability graphs are a proper subclass of clawfree graphs, whose chromatic symmetric functions are also conjectured by Stanley (and Gasharov) to be $s$-positive \cite{Sta98}, i.e., a non-negative linear expansion of Schur functions. For progress on this conjecture see \cite{Gas99,WW20} for instance. From now on we shall say that a graph $G$ is $e$-positive (resp. $s$-positive) if $X_G$ is $e$-positive (resp. $s$-positive). In another direction Stanley \cite{Sta95} asked whether two trees are isomorphic if and only if their chromatic symmetric functions are equal, which is later called Stanley's tree isomorphism conjecture. Several classes of trees are proved valid for this conjecture and there are also a lot of works related to it, see \cite{AdOZ23,ANRS24,CS20,HJ19,LS19,MMW11,NW99,OS14,WYZ24}.

In addition, Dahlberg, She, and van Willigenburg \cite{DSvW20} explored a new and interesting area of $e$-positivity and $s$-positivity of trees. They proved that for any $n$-vertex tree $T$, if the maximum degree $\Delta(T)$ of $T$ is at least $\log_2 n + 1$ then it is not $e$-positive, and if $\Delta(T) \ge \lceil n/2 \rceil$ then it is not $s$-positive. Wang and Wang \cite{WW20} further studied the positivity of spiders (trees having exactly one vertex of degree at least 3), and they proved the $e$-positivity of spiders $S(4m+2,2m,1)$ in collaboration with Tang \cite{TWW24}. Dahlberg, She, and van Willigenburg \cite{DSvW20} proposed the following conjecture, which directly relates $e$-positivity to the maximum degree of a tree.

\begin{conj}[\cite{DSvW20}]\label{conj-nonepos}
If $T$ is a tree with $\Delta(T) \ge 4$, then $T$ is not $e$-positive.
\end{conj}

Zheng \cite{Zhe22} proved that any tree $T$ with $\Delta(T) \ge 6$ is not-$e$-positive, and provided an explicit formula for calculating the $e$-coefficients of spiders. Tom \cite{Tom26} further proved that any tree with $\Delta(T) = 5$ is not $e$-positive and any spider $T$ with $\Delta(T) = 4$ is not $e$-positive. Most cases in their proofs use the following theorem of Wolfgang on $e$-positivity and the existence of connected partitions. Recall that given a graph $G$, a set partition $\{B_1,B_2,\ldots,B_k\}$ of its vertex set $V(G)$ is called a \textit{connected partition} of $G$ if for any $1 \le i \le k$ the subgraph of $G$ induced by the vertices in $B_i$ is connected. The \textit{type} of a connected partition $\{B_1,B_2,\ldots,B_k\}$ is defined as the integer partition (for detailed definition see Section \ref{sec-tool}) formed by the cardinalities $|B_1|,|B_2|,\ldots,|B_k|$.

\begin{thm}[\cite{Wol97}]\label{thm-epos-cp}
If a connected graph $G$ on $n$ vertices is $e$-positive, then it contains a connected partition of type $\lambda$ for any integer partition $\lambda$ of $n$.
\end{thm}

Unfortunately, this criterion sometimes fails. For example, for $m \ge 1$ the spider $S(6m,6m-2,1,1)$ is not $e$-positive but has a connected partition of every type, see \cite{Tom26}. Therefore, in this paper, we proceed to explore a quantitative version of Theorem \ref{thm-epos-cp}. Firstly, we give a formula for calculating the $e$-coefficients of any tree in terms of its connected partitions, which yields an equivalent condition for $e$-positivity. Moreover, the number of acyclic orientations with a fixed number of sinks can also be derived from this formula. Using these results, we present several necessary conditions for trees to be $e$-positive, which involves the numbers of connected partitions and acyclic orientations. As applications, we provide (local) structure characterizations of $e$-positive trees and then deduce the non-$e$-positivity of several classes of trees, which provide more evidence for Conjecture \ref{conj-nonepos}. Further, we make more discussions and computations on the trees with $\Delta=4$ and having a connected partition of every type. Based on these results, we propose several open problems, which may further advocate the final resolution of Conjecture \ref{conj-nonepos}.

This paper is organized as follows. In Section \ref{sec-tool} we introduce some basic results of symmetric functions, and develop the formulas on $e$-coefficients, connected partitions and acyclic orientations. In Section \ref{sec-appli} we present the necessary conditions mentioned above for trees to be $e$-positive, and we apply these conditions to prove the non-$e$-positivity of certain classes of trees. In Section \ref{sec-deg4}, we make more discussions on trees with $\Delta=4$ and having a connected partition of every type, and provide more explicit formulas. Finally, based on the above results and our computations, we come up with a list of open problems towards Conjecture \ref{conj-nonepos}.

\section{Formulas on \texorpdfstring{$e$}{e}-coefficients of trees}\label{sec-tool}

In this section, we shall introduce some basic definitions and derive several formulas on $e$-coefficients of trees, in which the numbers of connected partitions and acyclic orientations play an important role.

We first present some basic terminology in the theory of symmetric functions, see \cite{Mac15,RS15,StaEC2} for detailed information. Let $\mathbf{x} = \{x_1,x_2,\ldots\}$ be a set of countably infinite indeterminates. The algebra $\mathbb{Q}[[\mathbf{x}]]$ is defined as the commutative algebra of formal power series in $\mathbf{x}$ over the rational field $\mathbb{Q}$. The \textit{algebra of symmetric functions} $\Lambda_{\mathbb{Q}}(\mathbf{x})$ is defined to be the subalgebra of $\mathbb{Q}[[\mathbf{x}]]$ consisting of formal power series $f(\mathbf{x})$ of bounded degree and
\[
f(\mathbf{x})=f(x_1,x_2,\ldots) = f(x_{\omega(1)},x_{\omega(2)},\ldots)
\]
for any permutation $\omega$ of positive integers.

The bases of $\Lambda_{\mathbb{Q}}(\mathbf{x})$ are indexed by integer partitions. An \textit{integer partition} of $n$ is a sequence $\lambda = (\lambda_1,\ldots,\lambda_\ell)$ of positive integers such that $\lambda_1 + \lambda_2 + \cdots + \lambda_\ell = n$. The number $\ell=\ell(\lambda)$ is called the \textit{length} of $\lambda$. We write $\lambda \vdash n$ if $\lambda$ is an integer partition of $n$. By abuse of notation we write the partition $(n)$ as $n$, and we write $\lambda \vdash \mu$ or $\mu \dashv \lambda$ if $\lambda$ is a \textit{refinement} of $\mu$, i.e., $\lambda$ can be obtained by splitting the parts of $\mu$. For instance, $(2,2,1,1)$ is a refinement of $(3,2,1)$.

In this paper we mainly use two bases: the elementary symmetric functions and the power sum symmetric functions. For $\lambda = (\lambda_1,\ldots,\lambda_\ell)$, the \textit{elementary symmetric function} $e_{\lambda}$ is defined as
\[
e_{\lambda} = e_{\lambda_1}e_{\lambda_2}\cdots e_{\lambda_\ell},
\]
where for any $k \ge 1$
\begin{align*}
e_{k} = \sum_{1 \le j_1 < j_2 < \cdots < j_k} x_{j_1}x_{j_2} \cdots x_{j_k}. 
\end{align*}
Similarly,
For $\lambda = (\lambda_1,\ldots,\lambda_\ell)$, the \textit{power sum symmetric function} $p_{\lambda}$ is defined as
\[
p_{\lambda} = p_{\lambda_1}p_{\lambda_2}\cdots p_{\lambda_\ell},
\]
where for any $k \ge 1$
\begin{align*}
p_{k} = \sum_{i \ge 1} x_i^k. 
\end{align*}

Then we shall present the transition formulas between these two bases, for which we follow \cite{RS15} to introduce the concept of brick tabloids. At first, the \textit{Young diagram} of $\lambda = (\lambda_1,\ldots,\lambda_{\ell}) \vdash n$ is formed by $n$ left-justified boxes such that there are $\lambda_i$ boxes in the $i$-th row from bottom to top. Then a \textit{brick tabloid} of shape $\mu$ and content $\lambda$ is a Young diagram of $\mu$ with each row divided into ``bricks'' such that all the lengths of these bricks form $\lambda$ (here each brick contains one or more consecutive boxes). The set of all brick tabloids with shape $\mu$ and content $\lambda$ are denoted by $B_{\lambda,\mu}$, and by definition $B_{\lambda,\mu} \neq \varnothing$ if and only if $\lambda \vdash \mu$. For example, the four brick tabloids with shape $(4,2)$ and content $(2,2,1,1)$ are presented in Figure \ref{fig-brick-tab}.
\begin{figure}[htbp]
  \centering
\begin{tikzpicture}[scale = 1.5]
\draw[rounded corners=5pt]  (-4,4) rectangle (-3,3.5);
\draw[rounded corners=5pt]  (-5,4) rectangle (-4,3.5);
\draw[rounded corners=5pt]  (-5,4.5) rectangle (-4.5,4);
\draw[rounded corners=5pt]  (-4.5,4.5) rectangle (-4,4);
\draw[rounded corners=5pt]  (-2.5,4.5) rectangle (-1.5,4) node (v1) {};
\draw[rounded corners=5pt]  (-2.5,4) rectangle (-2,3.5);
\draw[rounded corners=5pt]  (-2,4) rectangle (-1.5,3.5);
\draw[rounded corners=5pt]  (v1) rectangle (-0.5,3.5);
\draw[rounded corners=5pt]  (0,4.5) rectangle (1,4);
\draw[rounded corners=5pt]  (0,4) rectangle (0.5,3.5);
\draw[rounded corners=5pt]  (0.5,4) rectangle (1.5,3.5);
\draw[rounded corners=5pt]  (1.5,4) rectangle (2,3.5);
\draw [rounded corners=5pt] (2.5,4.5) rectangle (3.5,4) node (v2) {};
\draw[rounded corners=5pt]  (2.5,4) rectangle (3.5,3.5);
\draw[rounded corners=5pt]  (v2) rectangle (4,3.5);
\draw[rounded corners=5pt]  (4,4) rectangle (4.5,3.5);
\end{tikzpicture}
  \caption{The brick tabloids in $B_{(2,2,1,1),(4,2)}$}\label{fig-brick-tab}
\end{figure}
Given a brick tabloid $T$, the \textit{weight} $w(T)$ is defined to be the product of the lengths of bricks ending each row of $T$ and $w(B_{\lambda,\mu}) = \sum_{T \in B_{\lambda,\mu}} w(T)$. For instance, the weight of the four brick tabloids in Figure \ref{fig-brick-tab} are 2, 4, 2, 2, respectively, and hence $w(B_{(2,2,1,1),(4,2)}) = 10$.

Moreover, an \textit{ordered brick tabloid} of shape $\mu$ and content $\lambda$ is a brick tabloid in $B_{\lambda,\mu}$ such that the bricks corresponding to $\lambda_1,\ldots,\lambda_{\ell}$ are labeled with $1,\ldots,\ell$ and these labels decrease in each row. The number of all ordered brick tabloids of shape $\mu$ and content $\lambda$ is denoted by $OB_{\lambda,\mu}$. The three ordered brick tabloids with shape $(4,2)$ and content $(2,2,1,1)$ are presented in Figure \ref{fig-obt}.
\begin{figure}[htbp]
  \centering
\begin{tikzpicture}[scale = 1.5]
\draw[rounded corners=5pt]  (-4,4) rectangle (-3,3.5);
\draw[rounded corners=5pt]  (-5,4) rectangle (-4,3.5);
\draw[rounded corners=5pt]  (-5,4.5) rectangle (-4.5,4);
\draw[rounded corners=5pt]  (-4.5,4.5) rectangle (-4,4);
\draw[rounded corners=5pt]  (-2.5,4.5) rectangle (-1.5,4) node (v1) {};
\draw[rounded corners=5pt]  (-2.5,4) rectangle (-2,3.5);
\draw[rounded corners=5pt]  (-2,4) rectangle (-1.5,3.5);
\draw[rounded corners=5pt]  (v1) rectangle (-0.5,3.5);
\draw[rounded corners=5pt]  (0,4.5) rectangle (1,4) node (v2) {};
\draw[rounded corners=5pt]  (0,4) rectangle (0.5,3.5);
\draw[rounded corners=5pt]  (0.5,4) rectangle (1,3.5);
\draw[rounded corners=5pt]  (v2) rectangle (2,3.5);
\node at (-4.75,4.25) {$4$};
\node at (-4.25,4.25) {$3$};
\node at (-4.25,3.75) {$2$};
\node at (-3.25,3.75) {$1$};
\node at (-2.25,3.75) {$4$};
\node at (-1.75,3.75) {$3$};
\node at (-0.75,3.75) {$2$};
\node at (-1.75,4.25) {$1$};
\node at (0.25,3.75) {$4$};
\node at (0.75,3.75) {$3$};
\node at (0.75,4.25) {$2$};
\node at (1.75,3.75) {$1$};
\end{tikzpicture}
  \caption{The ordered brick tabloids of shape (4,2) and content (2,2,1,1) }\label{fig-obt}
\end{figure}

Then the transition formulas between $e_{\lambda}$'s and $p_{\lambda}$'s are presented as follows:
\begin{align}
  e_{\mu} & =\sum_{\lambda \vdash \mu} (-1)^{n-\ell(\lambda)} \frac{OB_{\lambda,\mu}}{z_{\lambda}}p_{\lambda}, \label{eq-ep}\\
  p_{\mu} & =\sum_{\lambda \vdash \mu} (-1)^{n-\ell(\lambda)} w(B_{\lambda,\mu})e_{\lambda},\label{eq-pe}
\end{align}
where $z_{\lambda} = 1^{r_1}r_1!2^{r_2}r_2!\cdots$ if $\lambda$ has $r_1$ 1's, $r_2$ 2's, \ldots , etc.

Now we turn to the expansions of chromatic symmetric functions in terms of these two bases, for which Stanley proved the following results. Recall that a connected partition is a set partition $\{B_1,B_2,\ldots,B_k\}$ of $V(G)$ such that the induced subgraph of $G$ on $B_i$ is connected for any $1 \le i \le k$, and an acyclic orientation is an orientation of the edges of $G$ (also called arcs) such that there is no directed cycles.
\begin{thm}[\cite{Sta95}]\label{thm-exp-pe}
Let $G$ be a graph with vertex set $V$ and edge set $E$. Then
\[
X_G = \sum_{S \subseteq E} (-1)^{|S|} p_{\lambda(S)},
\]
where $\lambda(S)$ denotes the type of the connected partition induced by $S$, i.e., the parts of $\lambda(S)$ are the vertex sizes of the connected components of the spanning subgraph of $G$ with edge set $S$.

In particular, if $G$ is a tree, then
\[
X_G = \sum_{\lambda \vdash n} (-1)^{n-\ell(\lambda)} b_{\lambda}p_{\lambda},
\]
where $b_{\lambda}$ denotes the number of connected partitions of $G$ with type $\lambda$.

If $X_G = \sum_{\lambda \vdash n} c_{\lambda}e_{\lambda}$, then
\[
\mathrm{sink}(G,j) = \sum_{\substack{\lambda \vdash n \\ \ell(\lambda)=j}} c_{\lambda},
\]
where $\mathrm{sink}(G,j)$ denotes the number of acyclic orientations of $G$ with $j$ sinks (vertices with no arcs pointing out).
\end{thm}

Originally, Wolfgang proved Theorem \ref{thm-epos-cp} by means of Hopf algebra \cite{Wol97}, which build a qualitative relation between $e$-positivity and connected partitions. Now we are able to present a simple proof in the special case of trees.
\begin{prop}[\cite{Wol97}]\label{prop-tree-epos-cp}
If a tree $T$ on $n$ vertices is $e$-positive, then it contains a connected partition of type $\lambda$ for any integer partition $\lambda$ of $n$.
\end{prop}
\begin{proof}
Assume that $X_T = \sum_{\mu \vdash n} c_{\mu}e_{\mu}$ with $c_{\mu} \ge 0$ for any $\mu$. Then by \eqref{eq-ep} we have
\begin{align*}
X_T =& \sum_{\mu \vdash n} c_{\mu}\sum_{\lambda \vdash \mu} (-1)^{n-\ell(\lambda)} \frac{OB_{\lambda,\mu}}{z_{\lambda}}p_{\lambda} \\
=& \sum_{\lambda \vdash n} (-1)^{n-\ell(\lambda)}\left(\sum_{\mu \dashv \lambda}c_{\mu} \frac{OB_{\lambda,\mu}}{z_{\lambda}}\right)p_{\lambda},
\end{align*}
where all the terms in the internal summation have the same sign since $c_{\mu} \ge 0$. Taking $j=1$ in Theorem \ref{thm-exp-pe}, we have $[e_n]X_T = n$ since in any tree there is exactly one path connecting two vertices. Moreover, it is well-known that
\[
e_n = \sum_{\lambda \vdash n} \frac{(-1)^{n-\ell(\lambda)}}{z_{\lambda}} p_{\lambda},
\]
which can also be derived by \eqref{eq-ep}. Hence by Theorem \ref{thm-exp-pe} again we have
\[
b_{\lambda} = \sum_{\mu \dashv \lambda}c_{\mu} \frac{OB_{\lambda,\mu}}{z_{\lambda}} \ge c_{n}\frac{OB_{\lambda,n}}{z_{\lambda}} = \frac{n}{z_{\lambda}} > 0,
\]
and this completes the proof since $b_{\lambda}$ must be an integer.
\end{proof}

Now we derive the explicit formula for $e$-coefficients of all trees in terms of connected partitions.
\begin{thm}\label{thm-e-coe}
  For any tree $T$ the coefficient of $e_{\lambda}$ in $X_T$ is given by
\[
[e_{\lambda}]X_T = \sum_{\mu \dashv \lambda} (-1)^{\ell(\lambda)-\ell(\mu)} b_{\mu}w(B_{\lambda,\mu}).
\]
\end{thm}
\begin{proof}
By Theorem \ref{thm-exp-pe} and \eqref{eq-pe}, we have
\begin{align*}
  [e_{\lambda}]X_T =& [e_{\lambda}] \sum_{\mu \vdash n} (-1)^{n-\ell(\mu)} b_{\mu} p_{\mu} \\
  =& [e_{\lambda}] \sum_{\mu \vdash n} (-1)^{n-\ell(\mu)} b_{\mu} \sum_{\nu \vdash \mu} (-1)^{n-\ell(\nu)} w(B_{\nu,\mu})e_{\nu} \\
  =& \sum_{\mu \dashv \lambda} (-1)^{\ell(\lambda)-\ell(\mu)} b_{\mu}w(B_{\lambda,\mu}),
\end{align*}
as desired.
\end{proof}
\begin{rem}
By a result of Orellana and Scott \cite{OS14}, Zheng \cite{Zhe22} proved a formula reducing the chromatic symmetric functions of spiders to that of paths. Hence the $e$-coefficients of spiders can be calculated from the $e$-coefficients of paths, and the latter appeared implicitly in Stanley \cite{Sta95} and explicitly in Wolfe \cite{Wol98}. However, for trees with more irregular structures, reducing its chromatic symmetric function to that of paths would result in a messy expression.
\end{rem}

Using Theorem \ref{thm-e-coe}, we naturally obtain the following equivalent condition for $e$-positivity of trees.
\begin{thm}\label{thm-epos-equiv}
Let $T$ be a tree on $n$ vertices. Then $T$ is $e$-positive if and only if
\[
b_{\lambda} \ge \left\lceil \frac{1}{\prod_{i=1}^{\ell(\lambda)} \lambda_i} \sum_{\substack{\mu \dashv \lambda \\ \ell(\mu) < \ell(\lambda)}} (-1)^{\ell(\lambda)-\ell(\mu)-1}b_{\mu}w(B_{\lambda,\mu})\right\rceil
\]
for any partition $\lambda$ of $n$.
\end{thm}
\begin{proof}
By Theorem \ref{thm-e-coe}, $[e_{\lambda}]X_T \ge 0$ if and only if
\[
b_{\lambda}w(B_{\lambda,\lambda}) \ge  \sum_{\substack{\mu \dashv \lambda \\ \ell(\mu) < \ell(\lambda)}} (-1)^{\ell(\lambda)-\ell(\mu)-1}b_{\mu}w(B_{\lambda,\mu}).
\]
Moreover, we have $w(B_{\lambda,\lambda})= \prod_{i=1}^{\ell(\lambda)} \lambda_i$ since the unique tabloid with shape and content both being $\lambda$ has exactly one brick of $\lambda_i$ in row $i$ for each $1 \le i \le \ell(\lambda)$. Then the proof follows from the requirement that $b_{\lambda}$ must be an integer.
\end{proof}

Similarly, we have the following result on the numbers of acyclic orientations, which is valid for all trees.
\begin{thm}\label{thm-acyc-ori-all-j}
Let $T$ be a tree on $n$ vertices. Then it has exactly
\[
\mathrm{sink}(G,j) = \sum_{\substack{\lambda \vdash n \\ \ell(\lambda)=j}} \sum_{\mu \dashv \lambda} (-1)^{j-\ell(\mu)} b_{\mu}w(B_{\lambda,\mu})
\]
acyclic orientations of $j$ sinks.
\end{thm}
\begin{proof}
By Theorem \ref{thm-e-coe}, the $e$-expansion of $X_T$ is
\[
X_T = \sum_{\lambda \vdash n} \left( \sum_{\mu \dashv \lambda} (-1)^{\ell(\lambda)-\ell(\mu)} b_{\mu}w(B_{\lambda,\mu})\right) e_{\lambda}
\]
Then by Theorem \ref{thm-exp-pe} we have
\[
\mathrm{sink}(G,j) = \sum_{\substack{\lambda \vdash n \\ \ell(\lambda)=j}} \sum_{\mu \dashv \lambda} (-1)^{j-\ell(\mu)} b_{\mu}w(B_{\lambda,\mu}).
\]
This completes the proof.
\end{proof}

\section{Necessary conditions for \texorpdfstring{$e$}{e}-positivity of trees}\label{sec-appli}

In this section we proceed to apply Theorem \ref{thm-epos-equiv} and Theorem \ref{thm-acyc-ori-all-j} to derive necessary conditions for $e$-positivity of trees and deduce the non-$e$-positivity of several classes of trees. For the convenience of computation and illustration, we restrict ourselves to some special cases of these two theorems where the lengths of the partitions are relatively small. Under these restrictions we obtain more precise necessary conditions, and proceed to give more characterizations of $e$-positive trees. These results are stated for all trees and are particularly applicable for trees with $\Delta=3$ or $4$, which provide more evidence for Conjecture \ref{conj-nonepos}.

\begin{thm}\label{thm-n22}
Let $T$ be an $e$-positive tree with $n \ge 9$ vertices. Then
\[
b_{(n-4,2,2)} \ge \left\lceil \frac{b_{(n-2,2)} + b_{(n-4,4)}}{2} \right\rceil.
\]
\end{thm}
\begin{proof}
By Theorem \ref{thm-epos-equiv}, we have
\begin{align*}
b_{(n-4,2,2)} \ge &
\frac{b_{(n-2,2)}w(B_{(n-4,2,2),(n-2,2)})+b_{(n-4,4)}w(B_{(n-4,2,2),(n-4,4)}) - b_{n}w(B_{(n-4,2,2),n})}{(n-4) \cdot 2 \cdot 2}
\\
= & \frac{(2(n-4)+2\cdot2)b_{(n-2,2)} + 2(n-4)b_{(n-4,4)} - ((n-4)+2+2)b_{n}}{4(n-4)} \\
= & \frac{2(n-2)b_{(n-2,2)} + 2(n-4)b_{(n-4,4)} - n}{4(n-4)} \\
= & \frac{2(n-4)b_{(n-2,2)} + 2(n-4)b_{(n-4,4)} - (n-4b_{(n-2,2)})}{4(n-4)} \\
\ge & \frac{b_{(n-2,2)} + b_{(n-4,4)}}{2} - \frac{1}{4},
\end{align*}
where $w(B_{(n-4,2,2),(n-4,4)}) = 2(n-4)$ since when $n \ge 9$ there is a unique tabloid in $B_{(n-4,2,2),(n-4,4)}$, and in the last step we use $b_{(n-2,2)} \ge 1$ deduced from Proposition \ref{prop-tree-epos-cp}. Hence we have
\[
b_{(n-4,2,2)} \ge \left\lceil \frac{b_{(n-2,2)} + b_{(n-4,4)}}{2} \right\rceil
\]
since $b_{(n-4,2,2)}$ must be an integer.
\end{proof}

This simple condition allows us to give some characterizations of the local structure of $e$-positive trees. From now on we say that $T$ has a \textit{pendent (graph)} $G$ if there is a cut edge $f$ such that $T-f$ has two connected components with one component being isomorphic to $G$.

\begin{thm}\label{thm-one-p2}
If a tree $T$ satisfies one of the following conditions, then it is not $e$-positive:
\begin{itemize}
  \item[(1)] $T$ has no pendent $P_2$;
  \item[(2)] $T$ has exactly one pendent $P_2$ and at least one pendent claw $K_{1,3}$;
  \item[(3)] $T$ has exactly two pendent $P_2$'s and no pendent $P_4$. 
\end{itemize}
\end{thm}

\begin{proof}
Assume that $T$ has $n$ vertices. Note that the number of pendent $P_2$ in a tree is exactly the number of connected partitions of type $(n-2,2)$. Hence in case (1) $T$ is missing a connected partition of type $(n-2,2)$ and is not $e$-positive.

In case (2), we have $b_{(n-2,2)}=1$. Moreover, $T$ has at most one pendent $P_4$ since each pendent $P_4$ naturally has one pendent $P_2$. Then we have the following two subcases.
\begin{itemize}
  \item[(2.1)] $T$ has no pendent $P_4$. In this subcase $T$ is missing a connected partition of type $(n-4,2,2)$ and is not $e$-positive. Indeed, in any connected partition of type $(n-4,2,2)$ the unique pendent $P_2$ must form a block of size 2, and then the tree $T'$ obtained by deleting the pendent $P_2$ from $T$ have to admit a connected partition of type $(n-4,2)$. This means that $T'$ has a new pendent $P_2$, and these two $P_2$'s must be joined by some edge in $T$ and form a pendent $P_4$, leading to a contradiction.
  \item[(2.2)] $T$ has exactly one pendent $P_4$. By the argument in Subcase (2.1) it is straightforward to check that $b_{(n-4,2,2)}=1$. Moreover, this pendent $P_4$ and the pendent claw(s) yield $b_{(n-4,4)} \ge 2$, and hence
      \[
      \left\lceil \frac{b_{(n-2,2)} + b_{(n-4,4)}}{2} \right\rceil \ge \left\lceil \frac{3}{2} \right\rceil > 1 = b_{(n-4,2,2)},
      \]
      which implies that $T$ is not $e$-positive by Theorem \ref{thm-n22}.
\end{itemize}

In case (3) we have $b_{(n-2,2)}=2$. By the argument in Subcase (2.1),  the two pendent $P_2$'s cannot be joined by some edge in $T$ since $T$ has no pendent $P_4$. By the same reason they must form the two blocks of size 2 in any connected partition of $T$ with type $(n-4,2,2)$, and hence $b_{(n-4,2,2)}=1$. Finally, if $T$ is missing a connected partition of type $(n-4,4)$ then it is not $e$-positive, and if not we will get
\[
      \left\lceil \frac{b_{(n-2,2)} + b_{(n-4,4)}}{2} \right\rceil \ge \left\lceil \frac{3}{2} \right\rceil > 1 = b_{(n-4,2,2)}.
\]
which shows that $T$ is not $e$-positive by Theorem \ref{thm-n22}.
\end{proof}

\begin{exam}
\textup{We present an illustration of Theorem \ref{thm-one-p2} in Figure \ref{fig-pendent-p2}. The top two correspond to the trees in Subcase (2.2), where the left part is a pendent $P_4$, the right part is a pendent claw, and the subtree $T_1$ can be chosen to any tree as long as it has no pendent $P_2$ not joined to the pendent $P_4$ or claw.}

\textup{There are several different structures for trees in Case (3), and here for convenience we just take one class as an example, which is drawn as the bottom one in Figure \ref{fig-pendent-p2}. The left part of the graph contains the two pendent $P_2$'s and the subtree $T_2$ can be an arbitrary tree as long as it has no pendent $P_2$ not joined to the left part (and hence has no pendent $P_4$ and no other pendent $P_2$). We should remark that the two pendent $P_2$'s may also be joined to different vertices in $T$.}
\end{exam}
\begin{figure}[htbp]
  \centering
\begin{tikzpicture}[scale=1]

\fill (-4,3) circle (0.5ex);
\fill (-2,3) circle (0.5ex);
\draw (-5,3) -- (-4,3) -- (-3,3) -- (-2,3) -- (-1,3);

\fill (-3,3) circle (0.5ex);
\fill (-5,3) circle (0.5ex);

\node (v3) at (0,3) {$T_1$};

\fill (2,3) circle (0.5ex);
\draw (3,2) -- (3,3) -- (1,3);
\fill (3,3) circle (0.5ex);
\draw (3,3) -- (3,3) -- (4,3);
\fill (4,3) circle (0.5ex);
\fill (3,2) circle (0.5ex);

\fill (-4,1) circle (0.5ex);
\fill (-2,1) circle (0.5ex);
\draw (-4,1) -- (-3,1) -- (-2,1) -- (-1,1);
\draw (-2,1) -- (-2,0);

\fill (-3,1) circle (0.5ex);
\fill (-2,0) circle (0.5ex);

\node (v1) at (0,1) {$T_1$};

\fill (2,1) circle (0.5ex);
\draw (3,0) -- (3,1) -- (1,1);
\fill (3,1) circle (0.5ex);
\draw (3,1) -- (3,1) -- (4,1);
\fill (4,1) circle (0.5ex);
\fill (3,0) circle (0.5ex);

\fill (-4,-2) circle (0.5ex);
\fill (-2,-4) circle (0.5ex);
\fill (-2,-2) circle (0.5ex);
\draw (-4,-2) -- (-3,-2) -- (-2,-2) -- (-1,-2);
\draw (-2,-1) -- (-2,-2) -- (-2,-3) -- (-2,-4);

\fill (-3,-2) circle (0.5ex);
\fill (-2,-3) circle (0.5ex);
\fill (-2,-1) circle (0.5ex);

\node (v2) at (1,-2) {$T_2$};

\draw  (v1) ellipse (1 and 0.5);

\draw  (v2) ellipse (2 and 1);

\draw  (v3) ellipse (1 and 0.5);
\end{tikzpicture}
  \caption{Three classes of non-$e$-positive trees}\label{fig-pendent-p2}
\end{figure}

For the second part of this section, we apply Theorem \ref{thm-nece-cond-acyc} to give a necessary condition on the number of acyclic orientations with 2 sinks for $e$-positivity of trees.
\begin{thm}\label{thm-nece-cond-acyc}
Let $T$ be a tree with $n$ vertices and $l$ leaves. If $T$ is $e$-positive, then it has at least
\[
\sum_{k = \lceil \frac{n}{2} \rceil}^{n-2} k(n-k) + \frac{(2l-n)(n-1)}{2}
\]
acyclic orientations with 2 sinks.
\end{thm}
\begin{proof}
Taking $j=2$ in Theorem \ref{thm-acyc-ori-all-j}, we have
\begin{align*}
\mathrm{sink}(G,2) =& \sum_{\substack{\lambda \vdash n \\ \ell(\lambda)=2}} \sum_{\mu \dashv \lambda} (-1)^{2-\ell(\mu)} b_{\mu}w(B_{\lambda,\mu}) \\
=& \sum_{k=\lceil\frac{n}{2}\rceil}^{n-1} b_{(k,n-k)} w(B_{(k,n-k),(k,n-k)}) - b_{n} w(B_{(k,n-k),n}) \\
=& \sum_{k=\lceil\frac{n}{2}\rceil}^{n-1} k(n-k)b_{(k,n-k)} - w(B_{(k,n-k),n}),
\end{align*}
and by \eqref{eq-pe}
\[
w(B_{(k,n-k),n})=\begin{cases}
                     k+(n-k)=n, & \mbox{if $n$ is odd or $n$ is even and $k \neq \frac{n}{2}$;} \\
                     k=\frac{n}{2}, & \mbox{if $n$ is even and $k=\frac{n}{2}$}.
                   \end{cases}
\]
Hence
\begin{equation*}
  \mathrm{sink}(T,2) = \begin{cases}
                         \sum_{k=\frac{n+1}{2}}^{n-1} k(n-k)b_{(k,n-k)} - \frac{n-1}{2}n, & \mbox{if $n$ is odd;} \\
                         \sum_{k=\frac{n}{2}}^{n-1} k(n-k)b_{(k,n-k)} - \frac{n-1}{2}n, & \mbox{if $n$ is even},
                       \end{cases}
\end{equation*}
and we also write
\[
\mathrm{sink}(T,2) = \sum_{k=\lceil\frac{n}{2}\rceil}^{n-1} k(n-k)b_{(k,n-k)} - \frac{n(n-1)}{2}.
\]

If $T$ is $e$-positive, then by Propositon \ref{prop-tree-epos-cp} $T$ has a connected partition of every type, which means that $b_{\lambda} \ge 1$ for any $\lambda \vdash n$. Moreover, by definition the coefficient $b_{(n-1,1)}$ is exactly the number of leaves in $T$. Hence we have
\begin{align*}
\mathrm{sink}(T,2) =& \sum_{k=\lceil\frac{n}{2}\rceil}^{n-2} k(n-k)b_{(k,n-k)} + l(n-1) - \frac{n(n-1)}{2} \\
\ge & \sum_{k = \lceil \frac{n}{2} \rceil}^{n-2} k(n-k) + \frac{(2l-n)(n-1)}{2},
\end{align*}
as desired.
\end{proof}

Now we apply this result to caterpillars, for which the number of acyclic orientations can be explicitly calculated. We say that a tree $T$ is a \textit{caterpillar} if removing all leaves of $T$ yields a path $P$ (called the \textit{spine} of $T$). Moreover, if $P$ has $d$ vertices $v_1,\ldots,v_d$ arranged in order and there are $\alpha_i$ leaves adjacent to the $i$-th vertex, then we denote this caterpillar by $C(\alpha) = C(\alpha_1,\alpha_2,\ldots,\alpha_d)$, see Figure \ref{fig-caterpillar} for an example.
\begin{figure}[htbp]
  \centering
\begin{tikzpicture}

\fill (-4,3) circle (0.5ex);
\fill (-2,3) circle (0.5ex);
\draw (-4,2) -- (-4,3) -- (-3,3) -- (-2,3) -- (-1,3) -- (0,3);
\draw (-3,3) -- (-3,2);
\draw (-1,3) -- (-1,2);
\draw (-2.25,2) -- (-2,3) -- (-1.75,2);
\draw (0.25,2) -- (0,3) -- (-0.25,2);

\fill (-3,3) circle (0.5ex);
\fill (-4,2) circle (0.5ex);
\fill (-3,2) circle (0.5ex);
\fill (-2.25,2) circle (0.5ex);
\fill (-1.75,2) circle (0.5ex);
\fill (-1,3) circle (0.5ex);
\fill (0,3) circle (0.5ex);
\fill (-1,2) circle (0.5ex);
\fill (-0.25,2) circle (0.5ex);
\fill (0.25,2) circle (0.5ex);

\end{tikzpicture}
  \caption{The caterpillar $C(1,1,2,1,2)$}\label{fig-caterpillar}
\end{figure}
\begin{thm}\label{thm-acyc-caterpillar-nonepos}
Let $C(\alpha) = C(\alpha_1,\ldots,\alpha_d)$ be a caterpillar. If $C(\alpha)$ is $e$-positive, then
\begin{align*}
& \sum_{i=1}^d \binom{\alpha_i}{2} + \sum_{i=1}^d \alpha_i \sum_{j \neq i} |j-i| + \sum_{1 \le i < j \le d} \alpha_i\alpha_j(j-i+1) + j-i-1  \\
\ge &
\sum_{k = \lceil \frac{|\alpha|+d}{2} \rceil}^{|\alpha|+d-2} k(|\alpha|+d-k) + \frac{(|\alpha|-d)(|\alpha|+d-1)}{2},
\end{align*}
where $|\alpha| = \alpha_1 + \cdots + \alpha_d$.
\end{thm}
\begin{proof}
By Theorem \ref{thm-nece-cond-acyc} it suffices to prove
\begin{equation}\label{eq-num-acyc-calpha}
\mathrm{sink}(C(\alpha),2) = \sum_{i=1}^d \binom{\alpha_i}{2} + \sum_{i=1}^d \alpha_i \sum_{j \neq i} |j-i| + \sum_{1 \le i < j \le d} \alpha_i\alpha_j(j-i+1) + j-i-1 .
\end{equation}
To see this, we count the acyclic orientations of $C(\alpha)$ according to the following cases. 
\begin{itemize}
  \item[(1)] The two sinks are both leaves adjacent to the same vertex $v_i$ for some $1 \le i \le d$. Then for each two fixed sinks the acyclic orientation is unique and there are $\binom{\alpha_i}{2}$ choices of these two sinks for each $i$.
  \item[(2)] One sink is a leaf adjacent to $v_i$ and the other sink is a leaf adjacent to $v_j$ for some $1\le i < j\le d$. In this case the orientations of all edges except the ones between $v_i$ and $v_j$ on the spine  $P_d$ are fixed. For the edges between $v_i$ and $v_j$, we must select a middle vertex $v_k(i \le k \le j)$ such that the edges between $v_i$ and $v_k$ are oriented towards $v_i$ (and naturally towards $u_1$) and the edges from $v_k$ and $v_j$ are oriented towards $v_j$ (and naturally towards $u_2$). Then for each two fixed sinks the number of such orientations is $j-i+1$, and the number of choices of such sinks is $\alpha_i\alpha_j$ for any $1\le i < j\le d$.
  \item[(3)] One sink is a leaf adjacent to some $v_i$ and the other sink is exactly the vertex $v_j$. Then we have $i \neq j$ since otherwise there will be no acyclic orientations. Similar to Case (2), in this case only the edges between $v_i$ and $v_j$ are not fixed. But the edges $v_{j-1}v_j$ and $v_jv_{j+1}$ must be oriented towards $v_j$ since $v_j$ is a sink. Hence for each two fixed sinks the number of such orientations is $|j-i|$ and the number of choices for such sinks is $\alpha_i$ for any $j \neq i$.
  \item[(4)] The two sinks are just $v_i$ and $v_j$ for some $1 \le i < j \le d$. This case is also similar to Case (2) and Case (3), and the number of such acyclic orientations is $j-i-1$ for any $1 \le i < j \le d$.
\end{itemize}
Adding these terms together, we obtain \eqref{eq-num-acyc-calpha} and the proof follows.
\end{proof}

By this result we can show that many caterpillars are not $e$-positive. Now we give a simple example.
\begin{exam}
\textup{We consider the caterpillar $\alpha = (1,1,2,1,2)$ shown in Figure \ref{fig-caterpillar}, where $d = 5$ and $|\alpha| = 7$. By Theorem \ref{thm-acyc-caterpillar-nonepos} we have $\mathrm{sink}(C(\alpha),2) = 125$. However,
\[
\sum_{k = \lceil \frac{|\alpha|+d}{2} \rceil}^{|\alpha|+d-2} k(|\alpha|+d-k) + \frac{(|\alpha|-d)(|\alpha|+d-1)}{2} = 161 > 125,
\]
leading to a contradiction. Hence $C(1,1,2,1,2)$ is not $e$-positive.}
\end{exam}

\section{More discussions for trees with maximum degree 4}\label{sec-deg4}

In this section we proceed to give more explicit formulas for trees with $\Delta=4$ and having a connected partition of every type, which we will call \textit{CPET} in the remaining part of this article.

With the help of SageMath \cite{sage}, we verified Conjecture \ref{conj-nonepos} for trees with at most 21 vertices. As mentioned before, Zheng \cite{Zhe22} and Tom \cite{Tom26} have verified Conjecture \ref{conj-nonepos} for the trees with $\Delta \ge 5$ and spiders with $\Delta=4$. Therefore, here we only present all the non-spider CPET trees in Figure \ref{fig-CPET}, where $T_1,T_2$ have 17 vertices and $T_3,T_4$ have 19 vertices.
\vspace{0.5cm}
\begin{figure}[htbp]
  \centering
\begin{tikzpicture}[scale=0.9]

\fill (-7,1) circle (0.5ex);
\fill (-6,1) circle (0.5ex);
\fill (-5,1) circle (0.5ex);
\fill (-4,1) circle (0.5ex);
\fill (-3,1) circle (0.5ex);
\fill (-2,1) circle (0.5ex);
\fill (-1,1) circle (0.5ex);
\fill (-1.5,0) circle (0.5ex);
\fill (-0.5,0) circle (0.5ex);
\draw (-7,1) -- (-6,1) -- (-5,1) -- (-4,1) -- (-3,1) -- (-2,1) -- (-1,1) -- (0,1) -- (1,1) -- (2,1) -- (3,1) -- (4,1) -- (5,1) -- (6,1);
\draw (-0.5,0) -- (-1,1) -- (-1.5,0);
\draw (4,1) -- (4,0);
\fill (0,1) circle (0.5ex);
\fill (1,1) circle (0.5ex);
\fill (2,1) circle (0.5ex);
\fill (3,1) circle (0.5ex);
\fill (4,1) circle (0.5ex);
\fill (5,1) circle (0.5ex);
\fill (6,1) circle (0.5ex);
\fill (4,0) circle (0.5ex);

\fill (-7,-1) circle (0.5ex);
\fill (-6,-1) circle (0.5ex);
\fill (-5,-1) circle (0.5ex);
\fill (-4,-1) circle (0.5ex);
\fill (-3,-1) circle (0.5ex);
\fill (-2,-1) circle (0.5ex);
\fill (-1,-1) circle (0.5ex);
\fill (-1.5,-2) circle (0.5ex);
\fill (-0.5,-2) circle (0.5ex);
\draw (-7,-1) -- (-6,-1) -- (-5,-1) -- (-4,-1) -- (-3,-1) -- (-2,-1) -- (-1,-1) -- (0,-1) -- (1,-1) -- (2,-1) -- (3,-1) -- (4,-1) -- (5,-1) -- (6,-1);
\draw (-0.5,-2) -- (-1,-1) -- (-1.5,-2);
\draw (2,-1) -- (2,-2);
\fill (0,-1) circle (0.5ex);
\fill (1,-1) circle (0.5ex);
\fill (2,-1) circle (0.5ex);
\fill (3,-1) circle (0.5ex);
\fill (4,-1) circle (0.5ex);
\fill (5,-1) circle (0.5ex);
\fill (6,-1) circle (0.5ex);
\fill (2,-2) circle (0.5ex);

\fill (-8,-3) circle (0.5ex);
\fill (-7,-3) circle (0.5ex);
\fill (-6,-3) circle (0.5ex);
\fill (-5,-3) circle (0.5ex);
\fill (-4,-3) circle (0.5ex);
\fill (-3,-3) circle (0.5ex);
\fill (-2,-3) circle (0.5ex);
\fill (-1,-3) circle (0.5ex);
\fill (-2,-4) circle (0.5ex);
\draw (-8,-3) -- (-7,-3) -- (-6,-3) -- (-5,-3) -- (-4,-3) -- (-3,-3) -- (-2,-3) -- (-1,-3) -- (0,-3) -- (1,-3) -- (2,-3) -- (3,-3) -- (4,-3) -- (5,-3) -- (6,-3) -- (7,-3);
\draw (-2,-3) -- (-2,-4);
\fill (0,-3) circle (0.5ex);
\fill (1,-3) circle (0.5ex);
\fill (2,-3) circle (0.5ex);
\fill (3,-3) circle (0.5ex);
\fill (4,-3) circle (0.5ex);
\fill (5,-3) circle (0.5ex);
\fill (6,-3) circle (0.5ex);
\fill (7,-3) circle (0.5ex);
\fill (2.5,-4) circle (0.5ex);
\fill (3.5,-4) circle (0.5ex);
\draw (2.5,-4) -- (3,-3) -- (3.5,-4);

\fill (-7,-5) circle (0.5ex);
\fill (-6,-5) circle (0.5ex);
\fill (-5,-5) circle (0.5ex);
\fill (-4,-5) circle (0.5ex);
\fill (-3,-5) circle (0.5ex);
\fill (-2,-5) circle (0.5ex);
\fill (-1,-5) circle (0.5ex);
\fill (-1,-6) circle (0.5ex);
\fill (-1,-7) circle (0.5ex);
\fill (-1,-8) circle (0.5ex);
\draw (-7,-5) -- (-6,-5) -- (-5,-5) -- (-4,-5) -- (-3,-5) -- (-2,-5) -- (-1,-5) -- (0,-5) -- (1,-5) -- (2,-5) -- (3,-5) -- (4,-5) -- (5,-5) -- (6,-5);
\draw (-1,-5) -- (-1,-6) -- (-1,-7) -- (-1,-8);
\fill (0,-5) circle (0.5ex);
\fill (1,-5) circle (0.5ex);
\fill (2,-5) circle (0.5ex);
\fill (3,-5) circle (0.5ex);
\fill (4,-5) circle (0.5ex);
\fill (5,-5) circle (0.5ex);
\fill (6,-5) circle (0.5ex);
\fill (2.5,-6) circle (0.5ex);
\fill (1.5,-6) circle (0.5ex);
\draw (1.5,-6) -- (2,-5) -- (2.5,-6);
\node at (8,1) {$T_1$};
\node at (8,-1) {$T_2$};
\node at (8,-3) {$T_3$};
\node at (8,-5) {$T_4$};
\end{tikzpicture}
  \caption{All the non-spider CPET trees with $\Delta=4$ and $n \le 21$}\label{fig-CPET}
\end{figure}

Moreover, we found that each of the above four trees satisfies $e_{(3,2^{(n-3)/2})} < 0$, where $n$ is the number of vertices of the tree and by convention $2^k$ means there are $k$ parts equal to 2. This coefficient was also studied by Zheng \cite{Zhe22} when determining the $e$-positivity of certain spiders. These observations inspire us to derive the following general formula.

\begin{thm}\label{thm-sttt}
Let $T$ be a tree with $n=s+kt$ vertices with $s,t \ge 2$ being coprime positive integers. Then
\[
[e_{(s,t^k)}]X_T = \sum_{i=0}^k \sum_{\lambda \vdash (k-i)} (-1)^{k-\ell(\lambda)} t^{\ell(\lambda)}(s+it) b_{(s+it ,t\lambda)},
\]
where $t\lambda = (t\lambda_1,\ldots,t\lambda_{\ell})$.
\end{thm}
\begin{proof}
By Theorem \ref{thm-e-coe} we have
\[
[e_{(s,t^k)}]X_T = \sum_{\mu \dashv (s,t^k)} (-1)^{k+1-\ell(\mu)} b_{\mu} w(B_{(s,t^k),\mu}).
\]
If $(s,t^k)$ is a refinement of $\mu$, then there is exactly one part of $\mu$ which is not divisible by $t$, and its size has the form $s+it$ for some $0 \le i \le k$. In addition, the other parts of $\mu$ are all divisible by $t$, and hence we can arrange them into a partition $t\lambda$ for some $\lambda \vdash (k-i)$.

Then we proceed to calculate $w(B_{(s,t^k),(s+it,t\lambda)})$. At first, the brick of size $s$ must lie in the part of $s+it$ since $s$ and $t$ are coprime, and this brick must be inserted between or on one side of the $i$ bricks of size $t$. Moreover, the positions of the bricks with size $t$ are unique in the parts of $t\lambda$ since changing the order of bricks with the same size will not change the final brick tabloid. Hence we have
\[
w(B_{(s,t^k),(s+it,t\lambda)}) = s \cdot t^{\ell(\lambda)} + i \cdot t \cdot t^{\ell(\lambda)} = (s+it) t^{\ell(\lambda)}.
\]
In Figure \ref{fig-brick-stt} we provide an illustration of the form of the brick tabloids in $B_{(s,t^k),(s+it,t\lambda)}$, where all the unlabeled bricks have size $t$. It is straightforward to check that only the position of the brick with size $s$ varies.
\begin{figure}[htbp]
  \centering
\begin{tikzpicture}[scale = 1.2]
\draw[rounded corners=5pt]  (-4,4) rectangle (-3,3.5);
\draw[rounded corners=5pt]  (-5,4) rectangle (-4,3.5);
\draw[rounded corners=5pt]  (-5,4.5)  rectangle (-4,4);
\draw[rounded corners=5pt]  (-4,3.5)  rectangle (-2.5,3);
\draw[rounded corners=5pt]  (-5,3.5)  rectangle (-4,3);
\draw[rounded corners=5pt]  (-2.5,3.5)  rectangle (-1.5,3);
\node at (-3.25,3.25) {$s$};

\draw [rounded corners=5pt]  (-6.5,3.5) rectangle (-5.5,3);
\draw [rounded corners=5pt]  (-7.5,3.5) rectangle (-6.5,3);
\draw [rounded corners=5pt] (-9,3.5) rectangle (-7.5,3);
\draw [rounded corners=5pt]  (-9,4) rectangle (-8,3.5);
\draw [rounded corners=5pt] (-8,4) rectangle (-7,3.5);
\draw [rounded corners=5pt] (-9,4.5) rectangle (-8,4);
\node at (-8.25,3.25) {$s$};

\draw [rounded corners=5pt]  (-1,3.5) rectangle (0,3);
\draw [rounded corners=5pt]  (0,3.5) rectangle (1,3);
\draw [rounded corners=5pt]  (1,3.5) rectangle (2.5,3);
\draw [rounded corners=5pt]  (-1,4) rectangle (0,3.5);
\draw [rounded corners=5pt]  (0,4) rectangle (1,3.5);
\draw [rounded corners=5pt]  (-1,4.5) rectangle (0,4);
\node at (1.75,3.25) {$s$};
\end{tikzpicture}
  \caption{An illustration for the brick tabloids in $B_{(s,t^k),(s+it,t\lambda)}$}\label{fig-brick-stt}
\end{figure}

Summarizing the above results we obtain
\begin{align*}
[e_{(s,t^k)}]X_T =& \sum_{i=0}^k \sum_{\lambda \vdash (k-i)} (-1)^{k+1-(\ell(\lambda)+1)} b_{(s+it,t\lambda)} w(B_{(s,t^k),(s+it,t\lambda)}) \\
=& \sum_{i=0}^k \sum_{\lambda \vdash (k-i)} (-1)^{k-\ell(\lambda)} t^{\ell(\lambda)}(s+it) b_{(s+it ,t\lambda)},
\end{align*}
as desired.
\end{proof}
\begin{rem}
The restriction that $s,t$ are coprime may be relaxed in several ways. In fact, in the above proof the coprime condition is used to determine the partitions $\mu$ and the position of the brick of size $s$, which is also straightforward if $s>kt$ or $t=ms$ for some $m \ge 2$.
\end{rem}

We also derive the following necessary condition, which involves much fewer terms than the above theorem.
\begin{cor}\label{cor-reduced-formula}
Let $T$ be an $e$-positive tree with $n=s+kt$ vertices with $s,t \ge 2$ being coprime positive integers. Then
\[
\sum_{\lambda \vdash k} (-1)^{k-\ell(\lambda)} t^{\ell(\lambda)} b_{(s,t\lambda)} \ge 0
\]
\end{cor}
\begin{proof}
Note that $s+t$ is also coprime to $t$. By Theorem \ref{thm-sttt} we have
\begin{align*}
  [e_{(s,t^k)}]X_T =& \sum_{i=0}^k \sum_{\lambda \vdash (k-i)} (-1)^{k-\ell(\lambda)} t^{\ell(\lambda)}(s+it) b_{(s+it ,t\lambda)}, \\
  [e_{(s+t,t^{k-1})}]X_T =& \sum_{i=0}^{k-1} \sum_{\lambda \vdash (k-1-i)} (-1)^{k-1-\ell(\lambda)} t^{\ell(\lambda)}(s+t+it) b_{(s+t+it ,t\lambda)} \\
  =& \sum_{i=1}^{k} \sum_{\lambda \vdash (k-i)} (-1)^{k-1-\ell(\lambda)} t^{\ell(\lambda)}(s+it) b_{(s+it ,t\lambda)}.
\end{align*}
Substituting the second formula into the first one, we have
\[
0 \le [e_{(s,t^k)}]X_T = \sum_{\lambda \vdash k} (-1)^{k-\ell(\lambda)} s t^{\ell(\lambda)} b_{(s,t\lambda)} - [e_{(s+t,t^{k-1})}]X_T \le \sum_{\lambda \vdash k} (-1)^{k-\ell(\lambda)} s t^{\ell(\lambda)} b_{(s,t\lambda)},
\]
since $T$ is $e$-positive. Then the proof follows from $s>0$.
\end{proof}

\begin{exam}\label{ex-34444}
Using Theorem \ref{thm-sttt} or Corollary \ref{cor-reduced-formula}, we can show that the coefficients $e_{(3,2^7)}X_{T_1}$, $e_{(3,2^7)}X_{T_2}$, $e_{(3,2^8)}X_{T_3}$, $e_{(3,2^8)}X_{T_4}$ are all negative for the trees in Figure \ref{fig-CPET}. For the convenience of illustration we choose another coefficient and present the calculation. Taking the tree $T_4$ and $s=3,k=t=4$, we have
\begin{align*}
\sum_{\lambda \vdash k} (-1)^{k-\ell(\lambda)} t^{\ell(\lambda)} b_{(s,t\lambda)} =& 4^4 b_{(4,4,4,4,3)} - 4^3 b_{(8,4,4,3)} + 4^2 (b_{(12,4,3)}+b_{8,8,3}) - 4 b_{(16,3)} \\
=& 4^4 \cdot 1 - 4^3 \cdot 6 + 4^2 (6+2) - 4 \cdot 3 = -12 < 0,
\end{align*}
and hence $T$ is not $e$-positive.
\end{exam}

From Example \ref{ex-34444} (and also Theorem \ref{thm-epos-equiv}) one sees that the negativity of $e_{(4^4,3)}$ (resp. $e_{(3,2^8)}$) of this tree should be attributed to the fact that $b_{(4^4,3)}=1$ (resp. $b_{(3,2^8)}=1$). Roughly speaking, in Corollary \ref{cor-reduced-formula} the weight of $b_{(s,t^k)}$ is the highest ($=t^k$), and hence the sum is more likely to be negative if $b_{(s,t^k)}$ is small.

Naturally we guess that any (non-spider) CPET trees satisfies $b_{(3,2^{(n-3)/2})}=1$ and $e_{(3,2^{(n-3)/2})}<0$. Whether it is valid or not relies on the characterization of CPET trees, for which we shall present some observations and propose several open problems in the remaining part of this section. At first we recall Tom's result.
\begin{thm}[\cite{Tom26}]\label{thm-Tom}
If there is a vertex $v$ of a tree $T$ such that $T-v$ has 4 connected components such that each component has at least two vertices, then $T$ is missing a connected partition of some type and is not $e$-positive.
\end{thm}

By Figure \ref{fig-CPET} the CPET trees seem to satisfy a more stronger condition that deleting the vertex $v$ with degree 4 yields four connected components where two components have exactly one vertex. It is easy to see that if a tree $T$ on $n$ vertices has such a vertex $v$, then it is missing a partition of type $(2^{n/2})$ if $n$ is even, and $b_{(3,2^{(n-3)/2})} \le 1$ if $n$ is odd. Hence we have the following problem.

\begin{prob}\label{prob-ab11}
Let $T$ be a CPET tree with $\Delta(T)=4$. Is it true that there exists a vertex $v$ in $T$ with degree 4 is adjacent to at least two leaves? Moreover, is it true that there is exactly one vertex in $T$ with degree 4?
\end{prob}
We also want to investigate the numbers of vertices for CPET trees.
\begin{prob}\label{prob-prime}
Is it true that for any non-spider CPET tree with $\Delta=4$ the number its vertices is a prime number?
\end{prob}
This statement is not true for spiders since $S(12,10,1,1)$ has 25 vertices, whose CPET property was proved by Tom \cite{Tom26} as mentioned before. For general trees we also have a weaker statement.
\begin{prob}\label{prob-even}
Is it true that any CPET tree with $\Delta=4$ has an odd number of vertices?
\end{prob}
Finally, we propose the problem on the coefficient $e_{(3,2^{(n-3)/2})}$.
\begin{prob}\label{prob-322}
Is it true that $[e_{(3,2^{(n-3)/2})}]X_T<0$ for any CPET tree with $\Delta=4$ and having an odd number $n$ of vertices?
\end{prob}
We should remark that if both Problem \ref{prob-even} and Problem \ref{prob-322} have positive answers, then Conjecture \ref{conj-nonepos} will be true. Hence we believe the key point for solving this conjecture is the characterization of structure of CPET trees.

\vspace{0.5cm}

\noindent{\bf Acknowledgments.}
The author is supported by the National Natural Science Foundation of China (No. 12501462) and the Fundamental Research Funds for the Central
Universities (GK202207023).

\end{document}